\documentclass[a4paper,11pt]{article}

\usepackage[latin1]{inputenc}

\usepackage{amsmath}
\usepackage{amsfonts}
\usepackage{amssymb}
\usepackage{amstext    }
\usepackage{amsthm    }

\usepackage[a4paper]{geometry}
\usepackage{mathrsfs}

\usepackage{hyperref}

\usepackage{color}

\usepackage[T1]{fontenc}
\usepackage{graphicx}
\usepackage[all,cmtip]{xy}

\usepackage{epsf}
\usepackage{psfrag}

\usepackage{fancyhdr}

\usepackage{url}

\newtheorem{theorem}{Theorem}[section]
\newtheorem{lemma}[theorem]{Lemma}
\newtheorem{definition}[theorem]{Definition}
\newtheorem{proposition}[theorem]{Proposition}
\newtheorem{remark}[theorem]{Remark}
\newtheorem{corollary}[theorem]{Corollary}

\newcommand{\cali}[1]{\mathscr{#1}}

\numberwithin{equation}{section}

\newcommand{\dist}{{\rm dist}}
\newcommand{\ddc}{{dd^c}}

\newcommand{\supp}{{\rm supp}}

\newcommand{\Cc}{\cali{C}}
\newcommand{\Dc}{\cali{D}}

\newcommand{\Ic}{\cali{I}}
\newcommand{\Jc}{\cali{J}}

\newcommand{\Rc}{\cali{R}}

\newcommand{\Pb}{\mathbb{P}}
\newcommand{\Cb}{\mathbb{C}}

\usepackage[T1]{fontenc}
\usepackage[english]{babel}

\title{On chain recurrence classes of endomorphisms of $\Pb^k$}
\author{Johan Taflin}
\begin{document}
\maketitle

\begin{abstract}
We prove that the minimal chain recurrence classes of a holomorphic endomorphism of $\Pb^k$ have finitely many connected components. We also obtain results on arbitrary classes. These strong constraints on the topological dynamics in the phase space are all deduced from the associated action on a space of currents.
\end{abstract}


\section{Introduction}
The {\em chain recurrent set} is an important invariant set in topological dynamics which captures a large part of the non-trivial dynamics. For polynomial automorphisms of $\Cb^2,$ this set has been very well described by Bedford-Smillie in \cite{bs-2}. Our aim here is to initiate its study for holomorphic endomorphisms of $\Pb^k$ and to understand what kinds of restrictions the holomorphic assumption implies in that setting.

Let $f\colon\Pb^k\to\Pb^k$ be a holomorphic endomorphism. Recall that, following Conley \cite{conley}, for $\epsilon>0$ a sequence $(x_i)_{0\leq i\leq n}$ is called an \emph{$\epsilon$-pseudo-orbit} between $x$ and $y$ in $\Pb^k$ if $n\geq1,$ $x_0=x,$ $x_n=y$ and for all $0\leq i<n,$ $\dist(f(x_i),x_{i+1})<\epsilon.$ We say that $x\succ y$ if for all $\epsilon>0$ there exists an $\epsilon$-pseudo-orbit  between $x$ and $y.$ Then, the chain recurrent set is defined by
$$\Rc(f):=\{x\in \Pb^k\, |\, x\succ x\}.$$
This is a closed invariant set that contains the non-wandering set and thus all the periodic orbits. Moreover, $\succ$ is a preorder on $\Rc(f)$ and the equivalence classes associated to it are called the \emph{chain recurrence classes}: if $x\in\Rc(f)$ then its chain recurrence class $[x]$ consists of all $y\in\Rc(f)$ such that $x\succ y$ and $y\succ x.$ The relation $\succ$ becomes an order on the classes by saying that $[x]\succ[y]$ if $x\succ y.$

Our main result is the following theorem about minimal classes.
\begin{theorem}\label{th-main}
Let $f\colon\Pb^k\to\Pb^k$ be a holomorphic endomorphism of algebraic degree $d.$ If $K$ is a chain recurrence class for $f$ which is minimal with respect to $\succ$ then $ K$ has finitely many connected components.
\end{theorem}
In dimension $k=1,$ this result is a direct consequence of Montel's theorem. For $k=2,$ Forn{\ae}ss-Weickert \cite{fw-attractor} proved the stronger fact that a minimal chain recurrence class (referred to as {\em attractor} there, following Ruelle's terminology \cite{ruelle-book}) is either connected or an attracting periodic orbit (simply called {\em sink} in what follows). However, their main argument, which relies on pseudoconvexity, is specific to dimension $2.$ Our approach uses the dynamics induced by $f$ on some spaces of {\em currents} (see Section \ref{sec-current} for more details). Indeed, Theorem \ref{th-main} is an unexpected consequence of \cite{t-attractor} where other results on minimal chain recurrence classes were obtained.
\begin{theorem}[\cite{t-attractor}]\label{th-rappel}
Let $f$ be as in Theorem \ref{th-main}. The set of minimal chain recurrence classes of $f$ is at most countable. And if $K$ is such a class then there exist two integers $s\in\{0,1,\ldots, k\}$ and $n_0\geq1$ such that
\begin{itemize}
\item the topological entropy of $f_{|K}$ is $s\log d,$
\item $K$ is the disjoint union of minimal chain recurrence classes $K_1,\ldots,K_{n_0}$ for $f^{n_0},$
\item on each $K_i,$ there exists a measures $\nu_i$ which is mixing for $f^{n_0},$ of maximal entropy $s\log(d^{n_0})$ on $K_i$ and with at least $s$ positive Lyapunov exponents.
\end{itemize}
Moreover, if $s=0$ (i.e. if $f_{|K}$ has zero topological entropy) then $K$ is a sink.
\end{theorem}
What we prove in Theorem \ref{th-main} is that the sets $K_i$ are connected. Let us make some remarks about these two theorems. First, without the minimality assumption Theorem \ref{th-main} could be false. Indeed, the polynomial map $f(z)=z^2+1$ can be seen as an endomorphism of $\Pb^1$ and its Julia set is a chain recurrence class which is homeomorphic to a Cantor set. On the other hand, none of the results in Theorem \ref{th-main} and Theorem \ref{th-rappel} holds in smooth dynamics. There exist diffeomorphisms with uncountably many minimal chain recurrence classes and other ones $g$ with minimal classes which are {\em adding machine}, i.e. invariant set $K$ such that
\begin{itemize}
\item $K$ is a Cantor set,
\item the topological entropy on $K$ vanishes,
\item for each $n\geq1,$ $g^n_{|K}$ does not admits a mixing invariant measure.
\end{itemize}
See \cite{bonatti-diaz-maximal} for the abundance of diffeomorphisms with both these properties in the $C^1$-topology and \cite{gambaudo-vanstrien-tresser} for examples of analytic diffeomorphisms with the later property.

Observe that \cite{t-attractor} deals with ``minimal quasi-attractors'' which is synonymous to ``minimal chain recurrence classes'' (see Section \ref{sec-crs}). Here, we emphasize on the chain recurrent set as our goal in what follows is to give other results on non-minimal classes.

An easy consequence of classical equidistribution results is that the set of chain recurrence classes in $\Pb^k$ always admits a unique maximal element. Indeed, the dynamics on the support of an ergodic measure is topologically transitive so this support is included in a unique chain recurrence class. Moreover, an endomorphism $f$ of degree $d\geq2$ always admits a unique measure of maximal entropy $\mu$ \cite{briend-duval-carac}, called the {\em equilibrium measure}, and Forn{\ae}ss-Sibony \cite{fs-cdhd2} proved that Lebesgue almost every $y\in\Pb^k$ satisfies
$$\lim_{n\to\infty}\frac{1}{d^{kn}}\sum_{f^n(a)=y}\delta_a=\mu.$$
This implies directly the following result.
\begin{proposition}\label{prop-max}
If $[\mu]$ denotes the chain recurrence class associated to the equilibrium measure $\mu$ then every chain recurrence class $K$ satisfies $[\mu]\succ K.$
\end{proposition}
All the results above give some information on minimal or maximal classes. The remaining cases are related to the filtration of Julia sets. Recall that in \cite{fs-oka} Forn{\ae}ss-Sibony defined for each $l\in\{1,\ldots,k\}$ the {\em Julia set of order $l$}, $\Jc_l,$ of $f$ such that 
$$\Jc_k\subset\cdots\subset\Jc_1.$$
The largest one $\Jc_1$ corresponds to the standard Julia set, i.e. the complement of the {\em Fatou set} which is the largest open subset of $\Pb^k$ where the family $(f^n)_{n\geq1}$ is locally equicontinuous. On the other hand, the smallest one $\Jc_k$ is the support of the equilibrium measure $\mu.$ In general, little is known about the dynamics on $\Jc_1\setminus\Jc_k.$

\begin{theorem}\label{th-non-min-max}
Let $f\colon\Pb^k\to\Pb^k$ be an endomorphism of degree $d.$
\begin{itemize}
\item[\textbf{1)}] A chain recurrence class which is contained in the Fatou set is a sink.
\item[\textbf{2)}] If $K$ is a class which is neither maximal  nor minimal then there exist $s\in\{1,\ldots,k-1\}$ and an ergodic measure $\nu$ of entropy $s\log d$ supported in $\Jc_s\setminus\Jc_{s+1}$ such that the associated class $[\nu]$ satisfies $K\succ[\nu].$
\item[\textbf{3)}]  In particular, if the chain recurrent set of $f$ is not reduced to the union of the class $[\mu]$ and of sinks then there exists an ergodic measure with positive entropy supported in $\Jc_1\setminus\Jc_k.$
\item[\textbf{4)}] If $\Rc(f)$ is reduced to the union of $[\mu]$ with sinks then the Julia set $\Jc_1$ is contained in $\Rc(f).$
\end{itemize}
\end{theorem}

As the measures $\nu_i$ appearing in Theorem \ref{th-rappel}, $\nu$ in \textbf{2)}  corresponds to an {\em equilibrium measure associated to an attracting set} constructed in \cite{t-attractor}. Another way to formulate this point \textbf{2)} is that if a class is neither maximal nor minimal, it cannot admits only sinks as successors with respect to $\succ:$ one of them must come from a non-trivial ergodic measure. Observe that the self-maps of $\Pb^2$ in the {\em elementary Desboves familly} (see e.g. \cite{bdm-elliptic} or \cite{bianchi-t-desboves}) satisfy the assumption of \textbf{4)}.

As we will see in Section \ref{sec-crs}, each chain recurrent class is obtained as intersections of {\em attracting sets} and {\em repelling sets}. All the constraints above come from the ``attracting'' part, using our previous work \cite{t-attractor}. It seems that the ``repelling'' part is more flexible and it would be interesting to build examples of chain recurrent sets with unknown phenomena in $\Pb^k.$

\bigskip

\noindent\textbf{\textit{Organization of the paper.}} Section \ref{sec-crs} gives the basic definitions around attracting sets and some of their interplays with chain recurrence classes. In Section \ref{sec-current}, we recall the results from \cite{t-attractor} which are needed in the sequel. This section is technical but the idea is that the action of $f$ on a well-chosen set of currents in a neighborhood of an attracting set does not seem more complicated than the action on the set of probability measures in a neighborhood of a finite union of sinks. Finally, Section \ref{sec-proof} is devoted to the proof of Theorem \ref{th-main} and Theorem \ref{th-non-min-max}.

\bigskip

\noindent\textbf{\textit{Acknowledgements.}} This work has been supported by the ANR grant Fatou ANR-17-CE40-0002-01, the EIPHI Graduate School (contract ANR-17-EURE-0002) and Bourgogne-Franche-Comté Region.

\section{The chain recurrent set and attracting sets}\label{sec-crs}
We refer to the Introduction for the definition of the chain recurrent set. Here, we will explain its links with attracting sets.

Recall that if $f$ is an endomorphism of $\Pb^k$ then $A\subset\Pb^k$ is an {\em attracting set} if there exists a non-empty open set $U\subset\Pb^k,$ called a {\em trapping region}, such that $\overline{f(U)}\subset U$ and 
$$A=\cap_{n\geq0}f^n(U).$$
It implies that $A$ is a compact invariant set. The {\em basin of attraction} of $A$ is
$$B(A):=\cup_{n\geq0}f^{-n}(U),$$
and the {\em repelling set} associated to $A$ is
$$A^*:=\Pb^k\setminus B(A).$$
Finally, a {\em quasi-attractor} is a decreasing intersection of attracting sets.

One way to construct attracting sets and quasi-attractors is the following. For $x\in\Pb^k$ and $\epsilon>0,$ define
$$U_{x,\epsilon}:=\{y\in\Pb^k\,|\, \text{there exists an }\epsilon\text{-pseudo-orbit from }x\text{ to }y\}.$$
Such a set is open, not empty and satisfies $\overline{f(U_{x,\epsilon})}\subset U_{x,\epsilon}.$ It gives an attracting set 
\begin{equation}\label{eq-as}
A_{x,\epsilon}:=\cap_{n\geq0}f^n(U_{x,\epsilon})
\end{equation}
and a quasi attractor
\begin{equation}\label{eq-qa}
A_x:=\cap_{\epsilon>0}A_{x,\epsilon}.
\end{equation}
These notions are related to the chain recurrent set in the following way. First, it is easy to check that
$$\Rc(f)=\bigcap\left(A\cup A^*\right),$$
where the intersection is over all the attracting sets of $\Pb^k.$ Moreover, if $x,y\in\Rc(f)$ then $x\succ y$ is equivalent to $y\in A_x$ and we have the equality
$$[x]=A_x\setminus\left(\bigcup B(A)\right),$$
where the union is taken over all attracting sets $A$ such that $x\notin A.$ Hence, the class $[x]$ is minimal with respect to $\succ$ if and only if $[x]=A_x$ which is equivalent to $A_x$ minimal (with respect to inclusion) in the set of quasi-attractors.

\section{Dynamics in the space of currents}\label{sec-current}
The books \cite{de-book} and \cite{ds-lec} are good introductions to positive closed currents and the latter gives an overview of their uses in dynamics.

In order to define the mass of currents, let fix a Kähler form $\omega$ on $\Pb^k$ such that $\int_{\Pb^k}\omega^k=1.$ If $E$ is a subset of $\Pb^k$ and $p\in\{0,\dots,k\}$ then $\Cc_p(E)$ will designate the set of positive closed currents $S$ of bidegree $(p,p)$ such that $\supp(S)\subset E$ and $\|S\|:=\langle S,\omega^{k-p}\rangle=1.$

From now on, $f$ is an endomorphism of $\Pb^k$ of algebraic degree $d\geq2.$ There exists a special element in $\Cc_1(\Pb^k)$ with respect to the dynamics of $f,$ called the {\em Green current}, which can be defined by
$$T:=\lim_{n\to\infty}\frac{1}{d^n}f^{n*}\omega.$$
Forn{\ae}ss-Sibony \cite{fs-cdhd2} and Ueda \cite{ueda-fatou} proved that the support of $T$ is exactly the Julia set of $f.$ Moreover, for $l\in\{1,\dots,k\}$ the self-intersection $T^l$ is well-defined and its support is called the {\em Julia set of order $l,$} $\Jc_l:=\supp(T^l).$ This gives a filtration of totally invariant sets (i.e. invariant by $f^{-1}$)
$$\Jc_k\subset\cdots\subset\Jc_1.$$
In some sense (and still conjecturally for $l\notin\{1,k\}$), if $V$ is a generic algebraic set of pure codimension $l$ then $\Jc_l$ is the set where most of the volume of $f^{-n}(V)$ accumulates when $n$ goes to infinity. The case $l=k$ is particular. The measure $\mu:=T^k,$ called the {\em equilibrium measure} of $f,$ turns out to be the unique measure of maximal entropy for $f$ \cite{briend-duval-carac} and it has a central role in bifurcation theory \cite{bbd-bif}. The fact states above has been proven in this situation by Forn{\ae}ss-Sibony \cite{fs-cdhd2} with a pluripolar exceptional set $E,$ then improved by Briend-Duval \cite{briend-duval-carac} and finally, Dinh-Sibony obtained the following sharp result.
\begin{theorem}[\cite{ds-allupoly}]\label{th-equi}
There exists a proper algebraic subset $E\subset\Pb^k$ such that if $a\in\Pb^{k}\setminus E$ then
$$\lim_{n\to\infty}\frac{1}{d^{kn}}\sum_{y\in f^{-n}(a)}\delta_y=\mu.$$
\end{theorem}
As we have already said, this implies directly Proposition \ref{prop-max}.
\begin{proof}[Proof of Proposition \ref{prop-max}]
Let $E\subset\Pb^k$ be the proper algebraic subset given by Theorem \ref{th-equi} and let $x\in\Jc_k:=\supp(\mu).$ Let $K$ be a chain recurrence class, $z$ be a point of $K$ and $\epsilon>0.$ Since $\Pb^k\setminus E$ is dense, there is $a\in\Pb^k\setminus E$ such that $\dist(a,z)<\epsilon$ and by Theorem \ref{th-equi}, for $n\geq1$ large enough there is $y\in f^{-n}(a)$ with $\dist(f(x),y)<\epsilon.$ This gives an $\epsilon$-pseudo-orbit between $x$ and $z$ and thus, as $\epsilon>0$ was arbitrary, $[\mu]\succ K.$
\end{proof}

Relying on currents, Daurat \cite{daurat} defined the {\em dimension} of trapping regions, attracting sets and quasi-attractors.
\begin{definition}\label{def-daurat}
If $E$ is a trapping region, an attracting set or a quasi-attractor then $E$ is said to have dimension $s\in\{0,\dots,k\}$ and codimension $p:=k-s$ if
$$\Cc_{p}(E)\neq\varnothing\ \text{ and }\ \Cc_{p-1}(E)=\varnothing.$$
\end{definition}
Let $U\subset\Pb^k$ be a trapping region for $f$ and $A=\cap_{n\geq0}f^n(U)$ be the associated attracting set.
By the above definition, if $U$ has codimension $p$ then $f$ acts on the non-empty set $\Cc_p(U)$ by {\em normalized push-forward}
$$\Lambda:=\frac{1}{d^{(k-p)}}f_*.$$
This notation has an ambiguity as it depends on $p$ but we will only use it when the value of this codimension is clear. In this situation, it is natural to introduce the space of invariant current
$$\Ic_p(U):=\{S\in\Cc_p(U)\,|\,\Lambda S=S\}.$$
Another dynamically defined subset of $\Cc_p(U)$ will play an important role in Section \ref{sec-proof}. It has been introduced by Dinh in \cite{d-attractor} and it is given by
$$\Dc_p(U)=\left\{S\in\Cc_p(U)\,\hfill\vline\hfill\begin{minipage}{0.60\textwidth}\begin{center}there exists a sequence $(S_n)_{n\geq1}$ in $\Cc_p(U)$ such that\\$S=\Lambda^nS_n$\end{center}\end{minipage}\right\}.$$
It also corresponds to the set of all possible limit values of sequences of the form $(\Lambda^nS_n)_{n\geq1}$ where $(S_n)_{n\geq1}$ is a sequence in $\Cc_p(U).$
\begin{remark}\label{rk-ext}
In general, $\Ic_p(U)$ and $\Dc_p(U)$ do not coincide. For example, if $U$ is a small neighborhood of an attracting cycle $\{a,b\}$ of period $2$ then $\Ic_k(U)=\{(\delta_a+\delta_b)/2\}$ while $\Dc_k(U)$ is the set of convex combinations of $\delta_a$ and $\delta_b.$ In particular, here $(\delta_a+\delta_b)/2$ is not extremal in $\Dc_k(U).$
\end{remark}

The following lemma, obtained by Dinh \cite[Proposition 4.7]{d-attractor}, can be used to construct elements in $\Cc_p(U).$

\begin{lemma}\label{le existence courant}
Let $\chi$ be a positive smooth function in $\Pb^k.$ If $S$ is a current in $\Cc_{k-l}(\Pb^k)$ then the sequence $d^{-ln}(f^n)_*(\chi S)$ has bounded mass and each of its limit values is a positive closed $(k-l,k-l)$-current of $\Pb^k$ of mass $c:=\langle S\wedge T^{l},\chi\rangle.$
\end{lemma}
In particular, we used this result to obtain the following equivalent form of Definition \ref{def-daurat}.
\begin{proposition}[\cite{t-attractor}]
Let $E$ be a trapping region, an attracting set or a quasi-attractor. The dimension of $E$ is $s$ if and only if
$$E\cap\Jc_s\neq\varnothing\ \text{ and }\ E\cap\Jc_{s+1}=\varnothing.$$
\end{proposition}
The fact that $\Jc_{s+1}$ is disjoint from a dimension $s$ trapping region $U$ is crucial in the proof of the following theorem. In some sense, it says that there exists at least one element in $\Cc_p(U)$ with an attracting behavior.
\begin{theorem}[\cite{t-attractor}]
Let $U\subset\Pb^k$ be a trapping region of codimension $p.$ There exist a trapping region $D_\tau\subset U$ and a current $\tau$ in $\Cc_{p}(D_\tau)$ such that
$$\lim_{N\to\infty}\frac{1}{N}\sum_{n=0}^{N-1}\Lambda^nR=\tau$$
for all continuous currents $R$ in $\Cc_{p}(D_\tau).$
\end{theorem}
As this type of currents will play a central role in what follows, we coin the following definition.
\begin{definition}\label{def-attracting-current}
Let $V$ be a trapping region of codimension $p.$ A current $S\in\Cc_{p}(V)$ is \emph{attracting} on $V$ for $f$ if
$$\lim_{n\to\infty}\frac{1}{N}\sum_{n=0}^{N-1}\Lambda^nR=S,$$
for all continuous form $R$ in $\Cc_{p}(V).$ A current in $\Cc_{p}(\Pb^k)$ is an \emph{attracting current} if it is attracting on some trapping region of codimension $p.$
\end{definition}
As the simple example of an attracting cycle of period $2$ shows, it is not possible to remove the Cesàro mean in general. However, it turns out that the Cesàro mean is unnecessary if $f$ is replaced by an iterate.
\begin{theorem}\label{th-finite}
Let $U$ be a trapping region of codimension $p.$ There exists $n_0\geq1$ such that if we replace $f$ by $f^{n_0}$ the three following equivalent conditions hold.
\begin{itemize}
\item[\textbf{1)}] A current in $\Cc_p(U)$ is attracting for $f$ if and only if it is attracting for all iterates $f^n,$ $n\geq1.$
\item[\textbf{2)}] If $\tau\in\Cc_p(U)$ is attracting for $f$ on $V$ then
$$\lim_{n\to\infty}\Lambda^nR=\tau,$$
for all continuous form $R$ in $\Cc_{p}(V).$
\item[\textbf{3)}] The attracting currents in $\Cc_p(U)$ are extremal points in $\Dc_p(U).$
\end{itemize}
Moreover, the set of attracting currents in $\Cc_p(U)$ is finite.
\end{theorem}
The point \textbf{3)} was not explicitly states in \cite{t-attractor} but the proof is exactly the same than the extremality part of Theorem 3.12 and Corollary 3.21 there, using the first part of the proof of Theorem 3.35 i.e. if $\tau\in\Cc_p(U)$ is attracting on $V\subset U$ for all iterates of $f$ and if $\phi$ is a smooth $(s,s)$-form on $U$ such that $\ddc\phi\geq0$ then
$$\langle\tau,\phi\rangle=\max_{R\in \Dc_p(V)}\langle R,\phi\rangle.$$
Observe that, as it was point out in Remark \ref{rk-ext}, it is necessary to consider an iterate $f^{n_0}$ in order to obtain \textbf{3)}.

The finiteness part in Theorem \ref{th-finite} is indeed the key point of this result which gives \textbf{1)}, \textbf{2)} or \textbf{3)}. It is also this point which allows to obtain results about quasi-attractors.

\begin{corollary}\label{co-qa}
Let $K$ be a minimal quasi-attractor of dimension $s.$ There exists an integer $n_0\geq1$ such that if we replace $f$ by $f^{n_0}$ then $K$ splits into $n_0$ quasi-attractors $K=K_1\cup\cdots\cup K_{n_0}$ such that each $K_i$ is contained in a trapping region $U_{K_i}$ which supports a unique attracting current $\tau_i\in\Cc_{p}(K_i)$ with
$$\lim_{n\to\infty}\frac{1}{d^{ns}}(f^n)_*R=\tau_i$$
for all continuous currents $R$ in $\Cc_{p}(U_{K_i}).$
\end{corollary}
Finally, we will also use the following corollary about equilibrium measures associated to quasi-attractors.
\begin{corollary}\label{co-measure}
Let $A$ be a quasi-attractor of dimension $s$ and codimension $p:=k-s.$ There exists at least one attracting current $\tau$ in $\Cc_p(U)$ and the measure $\nu_\tau:=\tau\wedge T^s$ is an ergodic measure of entropy $s\log d$ supported in $\Jc_s\setminus\Jc_{s+1}$ and it has at least $s$ positive Lyapunov exponents. Moreover, if $n_0$ is the integer defined in Theorem \ref{th-finite} then $\nu_\tau$ has at most $n_0$ ergodic components with respect to $f^{n_0},$ each of which is mixing.
\end{corollary}
Observe that Theorem \ref{th-rappel} is a combination of Corollary \ref{co-qa} with Corollary \ref{co-measure}.

\section{Consequences on the chain recurrent set}\label{sec-proof}
In this section, we prove Theorem \ref{th-main} and Theorem \ref{th-non-min-max}. The former is a consequence of the following statement.
\begin{theorem}\label{th-tau-connexe}
Let $U$ be a trapping region of codimension $p.$ Let $\tau\in\Cc_p(U)$ be an attracting current for all iterates $f^n,$ $n\geq1.$ Then the support of $\tau$ is connected.
\end{theorem}
Notice that the assumption on $\tau$ is not really restrictive since, by Theorem \ref{th-finite}, any attracting current is an average of finitely many attracting currents for $f^{n_0}$ with this property.
\begin{proof}
Let $\Omega$ be an open subset of $U$ such that $X:=\Omega\cap\supp(\tau)$ is a non-empty compact set. Our aim is to prove that $X$ has to be equal to $\supp(\tau).$

Let $s:=k-p$ be the dimension of $U$ and recall that $\Lambda:=d^{-s}f_*.$ Let $0\leq\chi\leq1$ be a smooth function supported in $\Omega$ such that $\chi=1$ on a small neighborhood of $X.$ For $n\geq0,$ define the current
$$S_n:=\frac{(\chi\circ f^n)\tau}{\|\chi\tau\|}.$$
It is clear that these currents are positive and they are also of mass $1$ since
$$\|S_n\|=\langle S_n,T^s\rangle=\frac{\langle T^s\wedge\tau,\chi\circ f^n\rangle}{\langle T^s\wedge\tau,\chi\rangle}=1,$$
where the last equality comes from the invariance of the measure $\nu_\tau:=T^s\wedge\tau.$ Moreover,
$$\Lambda^n S_n=\frac{1}{d^{sn}}\frac{f^n_*(\chi\circ f^n\tau)}{\|\chi\tau\|}=\frac{\chi(\Lambda^n\tau)}{\|\chi\tau\|}=\frac{\chi\tau}{\|\chi\tau\|}=S_0.$$
Finally, we claim that each $S_n$ are closed. This implies that $S_0$ is in $\Dc_p(U).$ If $S_0\neq\tau$ (i.e. $\|\chi\tau\|<1$) then the same holds for
$$R_0:=\frac{(1-\chi)\tau}{\|(1-\chi)\tau\|}$$
and thus
$$\tau=\|\chi\tau\|S_0+\|(1-\chi)\tau\|R_0$$ 
which contradicts to fact that, by Theorem \ref{th-finite}, $\tau$ is extremal in $\Dc_p(U).$

It remains to prove that  $S_n$ is closed for each $n\geq0.$ Let $n\geq0$ and let $x\in\supp(S_n).$ By the definition of $S_n,$ this implies that $x\in\supp(\tau)$ and $x\in\supp(\chi\circ f^n).$ The first point gives, since $\supp(\tau)$ is invariant by $f,$ that $f^n(x)\in\supp(\tau)$ and the latter that $f^n(x)\in\supp(\chi),$ i.e. $f^n(x)\in\supp(\tau)\cap\supp(\chi)=X.$ But $\chi=1$ in a neighborhood of $X$ and $f^n$ is an open mapping thus $\chi\circ f^n=1$ in a neighborhood of $x.$ Hence, in this neighborhood $S_n$  coincides with a constant times $\tau$ and thus is closed.
\end{proof}
We can now deduce Theorem \ref{th-main} using that minimal chain recurrence classes correspond exactly to minimal quasi-attractors.
\begin{proof}[Proof of Theorem \ref{th-main}]
By Corollary \ref{co-qa}, if we replace $f$ by $f^{n_0}$ then $K$ is the union of $n_0$ quasi-attractors $K_i$ each of which supports a unique attracting current satisfying the assumption of Theorem \ref{th-tau-connexe}. We will show that each $K_i$ is connected and to simplify the notations, we assume that $n_0=1$ and $K=K_i.$

Hence, $K$ is a minimal quasi-attractor of dimension $s,$ which supports an attracting current $\tau$ such that $\supp(\tau)$ is connected. By definition, $K$ is the intersection of a decreasing family of attracting sets $(A_i)_{i\geq1}.$ Each attracting set $A_i$ has finitely many connected components that we denote by $A_{i,j}.$ Since $\supp(\tau)\subset K\subset A_i$ is connected, for each $i\geq1$ there exists $j(i)$ such that $\supp(\tau)\subset A_{i,j(i)}.$ A first observation is that, since the sequence $(A_i)_{i\geq1}$ is decreasing, the same holds for $(A_{i,j(i)})_{i\geq1}.$ On the other hand, the invariance of $\supp(\tau)$ implies that $f(A_{i,j(i)})=A_{i,j(i)}$ and it is easy to deduce from this that $A_{i,j(i)}$ is an attracting set for $f.$ Hence $K':=\cap_{i\geq1}A_{i,j(i)}$ is a quasi-attractor such that $\supp(\tau)\subset K'$ and thus it has dimension $s.$ Since $K'\subset K,$ the minimality of $K$ implies that $K'=K.$ Finally, $K'$ is a decreasing intersection of connected compact sets and thus it is also connected.
\end{proof}
\begin{remark}
Notice that Corollary \ref{co-qa} and Theorem \ref{th-tau-connexe} indeed hold under the weaker assumption that $K$ is minimal in the set of dimension $s$ quasi-attractors. Hence, the above proof gives that such $K$ has finitely many connected components.
\end{remark}
In order to prove Theorem \ref{th-non-min-max}, we need the following lemma which is a direct consequence of Lemma \ref{le existence courant} and will only be used with $l=1.$
\begin{lemma}\label{le-julia}
Let $l\in\{1,\dots,k\}.$ If $x\in\Jc_l$ then the dimension of the quasi-attractor $A_x$ defined by \eqref{eq-qa} is larger or equal to $l.$
\end{lemma}
\begin{proof}
Let $x$ be in $\Jc_l.$ Let $\epsilon>0$ and let $r>0$ such that if $\dist(y,z)<r$ then $\dist(f(y),f(z))<\epsilon.$ In particular, if $\dist(y,x)<r$ then $(x,f(y))$ is an $\epsilon$-pseudo-orbit and every limit value of $(f^n(y))_{n\geq1}$ belongs to $A_{x,\epsilon}.$ This implies that if $\chi$ is a smooth, positive cut-off function such that $\chi=1$ on the ball $B(x,r/2)$ and $\chi=0$ outside $B(x,r)$ then every limit value of $d^{-nl}f^n_*(\chi\omega^{k-l})$ is supported in $A_{x,\epsilon}.$ By Lemma \ref{le existence courant}, such limit value is a positive closed current of mass $c=\langle\omega^{k-l}\wedge T^l,\chi\rangle>0$ and thus $\Cc_{k-l}(A_{x,\epsilon})\neq\varnothing.$

Since this holds for each $\epsilon>0,$ it gives that $\Cc_{k-l}(A_x)\neq\varnothing$ and $\dim(A_x)\geq l.$
\end{proof}

\begin{proof}[Proof of Theorem \ref{th-non-min-max}]
Let $K$ be a chain recurrence class of $f$ and let $x$ be an element of $K.$ The proof of the point \textbf{1)} is elementary.

Assume that $K$ is included in the Fatou set. Since $K$ is compact, it must intersect a finite number of Fatou components and by invariance of $K,$ $f$ acts by permutation on them. Hence, by exchanging $f$ by an iterate, we can assume that $K$ is contained in an invariant Fatou component $\Omega.$ If $K$ is minimal with respect to $\succ$ then by Theorem \ref{th-rappel} it has to be equal to a sink. Otherwise, there exists another class $L$ such that $K\succ L.$ Since $K$ is invariant, it possesses two open neighborhoods, $N_1$ and $N_2,$ such that
$$\overline{N_1\cup f(N_1)}\subset N_2\ \text{ and }\ \overline{N_2}\subset\Omega\setminus L.$$
Hence, if $\epsilon>0$ is small enough then an $\epsilon$-pseudo-orbit from $K$ to $L$ must pass through $\overline{N_2}\setminus N_1$ and the compactness of this set implies that there is $y\in\overline{N_2}\setminus N_1$ such that $x\succ y.$ Moreover, $y$ is not in $K$ so there exists $\epsilon>0$ such that there is no $\epsilon$-speudo-orbit between $y$ and $K.$ Hence, $y$ is in the basin $B(A_{y,\epsilon})$ of the attracting set $A_{y,\epsilon}$ (defined by \eqref{eq-as}) and $K$ is disjoint from this basin. By connectedness, the Fatou component $\Omega$ must intersect the boundary of $B(A_{y,\epsilon}).$ This contradicts the fact that $\Omega$ is in the Fatou set.

Now, in order to prove \textbf{2)}, assume that $K$ is neither maximal nor minimal. In particular, it is not equal to the class $[\mu]$ associated to the equilibrium measure $\mu$ and, by \textbf{1)}, it must intersect the Julia set $\Jc_1.$ Hence, by Lemma \ref{le-julia}, the dimension $s$ of the associated quasi-attractor $A_x$ to $x\in K$ satisfies $1\leq s\leq k-1.$ Therefore, Corollary \ref{co-measure} gives a measure $\nu$ of entropy $s\log d$ supported on the intersection of $A_x$ with $\Jc_s\setminus\Jc_{s+1}.$ In particular, $K\succ[\nu].$ The proof of \textbf{3)} follows in the same way since we have only used that $K$ is different from $[\mu]$ or a sink.

Finally, assume that $\Rc(f)$ is a union of the class $[\mu]$ with sinks. Since each quasi-attractor of dimension $s$ gives rise to a chain recurrence class of entropy $s\log d$ by Corollary \ref{co-measure}, this assumption on $\Rc(f)$ implies that the only quasi-attractors of $f$ are of dimension $0$ or $k,$ i.e. sinks or $\Pb^k.$ Therefore, if $y$ is in $\Jc_1$ then by Lemma \ref{le-julia} the dimension of $A_y$ is larger or equal to $1$ thus $\dim(A_y)=k$ and $y\succ z$ for every $z$ in $\Pb^k.$ Hence, $y$ belongs to the maximal chain recurrence class $[\mu]\subset\Rc(f).$
\end{proof}


\bibliographystyle{alpha} 

\noindent {\footnotesize Johan Taflin}\\
{\footnotesize Universit\'e de Bourgogne-Franche-Comté}\\
{\footnotesize IMB, CNRS UMR 5584}\\
{\footnotesize Facult\'e des Sciences Mirande}\\
{\footnotesize 9 avenue Alain Savary}\\
{\footnotesize F-21000 Dijon, France}\\
{\footnotesize johan.taflin@u-bourgogne.fr}\\

\end{document}